\documentclass[12pt]{amsart}
\usepackage{amssymb, amsmath,latexsym}
\usepackage{enumerate}

\makeatletter
\@namedef{subjclassname@2010}{%
  \textup{2010} Mathematics Subject Classification}
\makeatother

\newtheorem{thm}{Theorem}[section]
\newtheorem{corollary}[thm]{Corollary}
\newtheorem{lem}[thm]{Lemma}
\newtheorem{fact}[thm]{Fact}

\theoremstyle{definition}
\newtheorem{defin}[thm]{Definition}
\newtheorem{rem}[thm]{Remark}

\numberwithin{equation}{section}

\begin{document}

%%%%%%%%%%% To ease editing, for IMPAN journals add:

\baselineskip=17pt

%%%%%%%%%%%

%% In the running head, replace first names by initials 
%% and give an abbreviation of the title.

\title[On the generalisation of the Hahn-Jordan decomposition]{On the generalisation of the Hahn-Jordan decomposition for real c\`{a}dl\`{a}g functions}

\author[R. M. \L{}ochowski]{Rafa\L \ M. \L ochowski}
\address{Department of Mathematics and Mathematical Economics, Warsaw School
of Economics, Madali\'{n}skiego 6/8, 02-513 Warszawa, Poland}
\email{rlocho@sgh.waw.pl}
\address{and}
\address{African Institute for Mathematical Sciences, 6 Melrose Road,
Muizenberg 7945, South Africa}
\email{rafal@aims.ac.za }
\date{}

\begin{abstract}
For a real c\`{a}dl\`{a}g function $f$ and a positive constant $c$ we find
another c\`{a}dl\`{a}g function, which has the smallest total variation
possible among all functions uniformly approximating $f$ with accuracy $c/2.$ The solution is expressed with the truncated variation, upward truncated
variation and downward truncated variation introduced in \cite
{Lochowski:2008} and \cite{Lochowski:2011}. They are are always finite even if the total variation of $f$ is infinite, and they may be viewed as the generalisation of the Hahn-Jordan decomposition for real c\`{a}dl\`{a}g functions. We also present partial results for more general functions.  
\end{abstract}



\maketitle

\section{Introduction}
The notion of a real-valued {\em signed measure} and its Hahn-Jordan decomposition plays fundamental role in the measure theory and the theory of integration. They are also related to the {\em upper, lower } and {\em total variations} of the signed measure \cite[Sect. IV.29]{Halmos:1950}. Generalisation to vector-valued measures is also possible. When the measurable space is the interval $[a;b],$ $-\infty < a < b < \infty,$ (with Borel $\sigma$-field of all measurable sets) instead of signed or vector-valued measures one may consider functions with finite total variation.

The total variation may be defined for any function $f:\left[a;b\right]\rightarrow E$ attaining values in a general metric space $E.$ Namely, when $\rho$ is the metric on $E$ we define the total variation of $f$ with the following formula
\[
TV(f,\left[a;b\right])=\sup_{n}\sup_{\pi_{n}}\sum_{i=1}^{n}\rho\left(f(t_{i}),f(t_{i-1})\right),
\]
where the second supremum is over all partitions $\pi_{n}=\left\{ a\leq t_{0}<t_{1}<...<t_{n}\leq b\right\} .$ 

In general, the total variation of $f$ may be (and in many important cases is) infinite. For example, almost all paths of a standard Brownian motion, which is widely used in stochastic modeling and optimisation, are continuous functions with infinite total variation on any interval $[0;t],$ $t>0.$ This fact was arguably the main reason for the introduction of the It\^{o} stochastic integral.  

However, after imposing some mild regularity conditions on $f$ we will easily find functions approximating $f$ with arbitrary accuracy and having finite total variation, even if the total variation of $f$ is infinite. Obviously, the better approximation is, the greater is the total variation of the approximating function. Let us fix $\ c>0.$ The natural question arises, what is the greatest lower bound for total variation of function $g:[a;b]\rightarrow E,$ uniformly approximating $f$ with accuracy $c/2>0,$ i.e. what is 
\[
\inf_{g\in B\left(f,c/2\right)}TV\left(g,\left[a;b\right]\right),
\]
where $B\left(f,d\right)$ denotes the ball 
\begin{align*}
B\left(f,d\right): & =\left\{ g:\left[a;b\right]\rightarrow E: \sup_{t\in[a;b]}\rho\left(f\left(t\right),g\left(t\right)\right)\leq d\right\} .
\end{align*}
The immediate bound from below for $\inf_{g\in B\left(f,c/2\right)}TV\left(g,[a;b]\right)$
reads as 
\begin{equation}
\inf_{g\in B\left(f,c/2\right)}TV\left(g,[a;b]\right)\geq\sup_{n}\sup_{\pi_{n}}\sum_{i=1}^{n}\max\left\{ \rho\left(f(t_{i}),f(t_{i-1})\right)-c,0\right\} \label{eq:lower_bound}
\end{equation}
and follows directly from the triangle inequality 
\begin{align*}
\rho\left(g(t_{i}),g(t_{i-1})\right) & \geq\rho\left(f(t_{i}),f(t_{i-1})\right)-\rho\left(f(t_{i}),g(t_{i})\right)-\rho\left(f(t_{i-1}),g(t_{i-1})\right)\\
 & \geq\rho\left(f(t_{i}),f(t_{i-1})\right)-c.
\end{align*}
We will call the quantity on the right hand side of (\ref{eq:lower_bound}), i.e.
\[
\sup_{n}\sup_{\pi_{n}}\sum_{i=1}^{n}\max\left\{ \rho\left(f(t_{i}),f(t_{i-1})\right)-c,0\right\} 
\]
\emph{truncated variation of the function $f$ at the level $c$} and denote it by $TV^{c}\left(f,[a;b]\right);$ it was first introduced in \cite{Lochowski:2008}.

The just obtained lower bound for $\inf_{g\in B\left(f,c/2\right)}TV\left(g,[a;b]\right)$ may also be infinite but from inequality (\ref{eq:lower_bound}) it follows that it is finite for any $c>0$ iff the function $f$ is an uniform limit of finite variation functions.  We prove this fact and identify the family of such functions in Section \ref{Preliminaria} (Fact \ref{regulated}). 

The family of real c\`{a}dl\`{a}g functions, i.e. right-continuous functions with left limits, will be of our special interest since c\`{a}dl\`{a}g functions with finite total variations correspond naturally to finite signed measures on the interval $(a;b].$
Moreover, in this paper we will show that for c\`{a}dl\`{a}g $f,$ $E=\mathbb{R}$ with the standard Euclidean metric $\rho\left(x,y\right)=\left|x-y\right|$ and any $c>0$
we have in fact equality, i.e. 
 \begin{equation}
\inf_{\left\Vert g-f\right\Vert _{\infty}\leq c/2}TV\left(g,[a;b]\right)=TV^{c}\left(f,[a;b]\right),\label{eq:tvequality}
\end{equation}
where $g:[a;b]\rightarrow \mathbb{R},$ $\left\Vert g-f\right\Vert _{\infty}=\sup_{t\in[a;b]}\left|g\left(t\right)-f\left(t\right)\right|.$
Morever, there exist such a c\`{a}dl\`{a}g function $f^{c}:[a;b]\rightarrow\mathbb{R}$
that 
\[
\left\Vert f^{c}-f\right\Vert _{\infty}\leq c/2\mbox{ and \ensuremath{TV\left(f^{c},[a;b]\right)=TV^{c}\left(f,[a;b]\right)}}.
\]
\begin{rem}
In general, the function $f^c$ is not unique, however, 
imposing stronger condition that  $\left\| f^{c} - f \right\| _{\infty }\leq c/2$ and for any $s\in \left( a;b\right]$ 
\begin{equation}
TV\left( f^{c},\left[ a;s\right] \right) =TV^{c}\left( f,\left[ a;s\right]
\right), \label{uniq}
\end{equation}
we will obtain that the function $f^{c}$ exists and is uniquely determined for any $c\leq \sup_{s,u\in \left[ a;b\right] }\left| f\left( s\right) -f\left( u\right) \right| $ (cf. Corollary \ref{cor1}). 
\end{rem}
\begin{rem}
The natural question appears if the truncated variation is attainable lower bound for $\inf_{g\in B\left(f,c/2\right)}TV\left(g,[a;b]\right)$ for functions with values in other metric spaces, but the answer to this problem is not known to the author. In \cite[Lemma 9]{Vladimirov:2000} it was proven that if $f$ is continuous and $E$ is a general, multidimensional (and complete metric) space then $\inf_{g\in B\left(f,c/2\right)}TV\left(g,[a;b]\right)$ is attained for some function $g_0,$ however, authors do not identify this quantity as the truncated variation. The proof of \cite[Lemma 9]{Vladimirov:2000} works for any c\`{a}dl\`{a}g function $f.$
\end{rem}
Since for $E=\mathbb{R}$ with $\rho\left(x,y\right)=\left|x-y\right|$ the total variation depends only on the increments of the function, in this case 
a more natural problem would be the following.
For a c\`{a}dl\`{a}g function $f:\left[ a;b\right] \rightarrow \mathbb{R}$
and $c>0$ find 
\begin{equation*} 
\inf \left\{ TV\left( f+h,\left[ a;b\right] \right) :\left\| h\right\|
_{osc}\leq c\right\} ,
\end{equation*}
where for $h:\left[ a;b\right] \rightarrow \mathbb{R},$ $\left\| h\right\| _{osc}:=\sup_{s,u\in \left[ a;b\right] }\left|
h\left( s\right) -h\left( u\right) \right| .$ Note that $\left\| .\right\|
_{osc}$ is a norm on the classes of bounded functions which differ by a
constant.

Solution to this problem is the same as the solution to the preceding
problem, i.e. 
\begin{equation}
\inf \left\{ TV\left( f+h,\left[ a;b\right] \right) :\left\| h\right\|
_{osc}\leq c\right\} =TV^{c}\left( f,\left[ a;b\right] \right) 
\label{tvceq}
\end{equation}
and one of the optimal representatives of the class of functions for which equality (\ref{tvceq}) is attained is $h^c=f^c - f$. To this class also belongs some $h^{0,c},$ such that $h^{0,c}\left( a\right) =0.$ 
We will prove that $f^{0,c}=f+h^{0,c} - f(a)$ is a c\`{a}dl\`{a}g function with
possible jumps only in the points where the function $f$ has jumps and that it
may be represented in the following form 
\begin{equation}
f^{0,c}\left( s\right) = UTV^{c}\left( f;\left[ a;s\right]
\right) -DTV^{c}\left( f;\left[ a;s\right] \right), \label{utv-dtv}
\end{equation}
where
\begin{equation}
UTV^{c}\left( f,\left[ a;s\right] \right) :=\sup_{n}\sup_{a\leq
t_{0}<t_{1}<...<t_{n}\leq s}\sum_{i=1}^{n}\max \left\{ f\left( t_{i}\right)
-f\left( t_{i-1}\right) -c,0\right\} ,  \label{utv:def}
\end{equation}
\begin{equation}
DTV^{c}\left( f,\left[ a;s\right] \right) :=\sup_{n}\sup_{a\leq
t_{0}<t_{1}<...<t_{n}\leq s}\sum_{i=1}^{n}\max \left\{ f\left(
t_{i-1}\right) -f\left( t_{i}\right) -c,0\right\}.  \label{dtv:def}
\end{equation}
The functionals $UTV^{c}\left( f,\left[ a;s\right] \right)$ and $DTV^{c}\left( f,\left[ a;s\right] \right)$ are non-decreasing functions of $s$ and are called {\em upward} and {\em downward truncated variations} of the function $f$ of order $c$ on the interval $[a;s]$ respectively. They were first introduced in \cite{Lochowski:2011} with a bit different formulae, equivalent with (\ref{utv:def}) and (\ref{dtv:def}).  

Finally, for $s \in (a;b]$ we will show the following equality 
\begin{equation}
TV\left(f^{0,c},\left[ a;s\right] \right) = TV^{c}\left( f,\left[ a;s\right] \right) =UTV^{c}\left( f,\left[ a;s\right]
\right) +DTV^{c}\left( f,\left[ a;s\right] \right).  \label{sumutvdtv}
\end{equation}
The equalities (\ref{utv-dtv}) and (\ref{sumutvdtv}) give the Hahn-Jordan decomposition of the finite signed measure, induced by the function $f^{0,c}$ (or by the function $f^c$). This measure assigns to any interval $\left(a_1,b_1 \right] \subset (a;b]$ the number
\[
\mu \left(a_1,b_1 \right] = f^{0,c}\left(b_1\right) - f^{0,c}\left(a_1\right)
\]
and we have 
\[
\mu \left(a_1,b_1 \right] =  \mu_+ \left(a_1,b_1 \right] - \mu_- \left(a_1,b_1 \right],
\]
where 
\[
\mu_+ \left(a_1,b_1 \right] = UTV^{c}\left( f,\left[ a;b_1\right]
\right) - UTV^{c}\left( f,\left[ a;a_1\right]
\right), 
\]
\[
\mu_- \left(a_1,b_1 \right] = DTV^{c}\left( f,\left[ a;b_1\right]
\right) - DTV^{c}\left( f,\left[ a;a_1\right]
\right).
\]
However, since $c>0$ is arbitrary, the equalities (\ref{utv-dtv}) and (\ref{sumutvdtv}) also may be viewed as the generalisation of the Hahn-Jordan decomposition for any real c\`{a}dl\`{a}g function $f$.

\begin{rem}
The truncated variation and its decomposition into the sum of upward and downward truncated variations appeared naturally when the uniform approximation of the c\`{a}dl\`{a}g function $f$ with finite variation functions was considered. 
The truncated variation is obtained by the composition of increments of $f$ with a convex function $\varphi(\cdot)=\left(\left| \cdot \right| -c \right)_+.$ 
Naturally, for any Young function (convex, non-decreasing, non-constant and vanishing at $0$) $\varphi:[0;+\infty)\rightarrow \mathbb{R}$ the notion of $\varphi$-variation defined as 
\[
TV^{\varphi}\left(f, [a;b] \right):= \sup_n \sup_{a\leq t_0 < t_1 < \ldots < t_n \leq b} \sum_{i=1}^{n} \varphi\left(\left|f(t_i) - f(t_{i-1})\right| \right)
\]
is of meaning. More on $\varphi$-variation may be found in \cite[Chapt. 3]{Norvaisa:2011}. The authors of \cite{Norvaisa:2011} consider only the case when $\varphi$ is strictly increasing, since for such $\varphi,$ corresponding  $\varphi$-variation leads to interesting estimates for integrals (generalisations of the Love-Young inequality). 

However, for any Young function $\varphi:[0;+\infty)\rightarrow \mathbb{R}$ the functional 
\[
\Vert f \Vert_{(\varphi)}:= \inf\left\{C>0 : TV^{\varphi}\left(f/C, [a;b] \right) \leq 1 \right\}
\]
is a seminorm on the space of such functions $f:[a;b] \rightarrow \mathbb{R}$ that $TV^{\varphi}\left(f/C, [a;b] \right) < +\infty $ for some $C>0$ (cf. \cite[Chapt. 3, proof of Theorem 3.7]{Norvaisa:2011}). $\Vert \cdot \Vert_{(\varphi)}$ is also a norm on the space of classes of abstraction of such functions, differing by a constant. For two Young functions $\varphi$ and $\psi,$ $\Vert \cdot \Vert_{(\varphi)}$ and $\Vert \cdot \Vert_{(\psi)}$ are equivalent when the ratio of the right-continuous inverse functions $\varphi^{-1}/\psi^{-1}$ is separated from $0$ and from $+\infty.$ Let us notice however, that not for every Young function $\varphi$ the corresponding $\varphi$-variation may be decomposed into the sum of upward and downward $\varphi$-variation. To see this consider the example. Let $\varphi$ be such that $\varphi(0)= \varphi(1) = 0,$ $\varphi(2)= 1,$ $\varphi(3)= 2$ and $\varphi(4)= 6;$ let $f$ be increasing on the interval $[0;1],$ decreasing on the interval $[1;2],$ and increasing on the interval $[2;3]$ with $f\left([0;1] \right)=[0;3],$ $f\left([1;2] \right)=[1;3]$ and $f\left([2;3] \right)=[1;4].$  Defining  
\[
UTV^{\varphi}\left(f, [a;b] \right) := \sup_n \sup_{a\leq t_0 < t_1 < \ldots < t_n \leq b} \sum_{i=1}^{n} \varphi\left(\left(f(t_i) - f(t_{i-1})\right)_{+}\right),
\] 
\[
DTV^{\varphi}\left(f, [a;b] \right) := \sup_n \sup_{a\leq t_0 < t_1 < \ldots < t_n \leq b} \sum_{i=1}^{n} \varphi\left(\left(f(t_i) - f(t_{i-1})\right)_{-}\right),
\] 
we have $TV^{\varphi}\left(f, [0;3] \right) = 6,$ $UTV^{\varphi}\left(f, [0;3] \right) = 6,$ and $DTV^{\varphi}\left(f, [0;3] \right) = 1,$ thus 
\[TV^{\varphi}\left(f, [0;3] \right) < UTV^{\varphi}\left(f, [0;3] \right) + DTV^{\varphi}\left(f, [0;3] \right).\] 

These and other properties of $TV^{\varphi}$ for general Young function $\varphi$ will be the subject of further investigation.

\end{rem}

\begin{rem}
Since we deal with c\`{a}dl\`{a}g functions, a more natural setting of the first problem would be the investigation of 
\begin{equation*}
\inf \left\{ TV\left( g,\left[ a;b\right] \right) : g \text{ - c\`{a}dl\`{a}g}, d_{D}(g,f)\leq c/2\right\}, 
\end{equation*}
where $d_{D}$ denotes the Skorohod metric (cf. \cite[Chapt. 3]{Billingsley}). However, the total variation does not depend on the (continuous and strictly increasing) transformations of the argument and for $E = \mathbb{R}$ with $\rho\left(x,y\right)=\left|x-y\right|$ the function $f^c$ minimizing $ TV\left( g,\left[ a;b\right] \right)$ appears to be a c\`{a}dl\`{a}g one, hence solutions of both problems coincide in this case.
\end{rem}

Let us comment on the organisation of the paper. In the next section we deal with functions attaining values in general metric spaces and prove Fact \ref{regulated}. In the third section we deal with real c\`{a}dl\`{a}g functions -  introduce some necessary definitions and notation, and present the construction of the functions $f^{c}$ and $f^{0,c}$ of the first and the second problem. In the fourth section we establish the connection between $f^{0,c}$ and truncated variation, upward truncated variation and downward truncated variation. In the last section we summarise some other general properties of (upward, downward) truncated variation, e.g. we will show that for any real c\`{a}dl\`{a}g function $f,$ $TV^{c}\left( f,\left[ a;b\right] \right) $ is a continuous, convex and decreasing function of the parameter $c>0.$ 

\section{Truncated variation of functions attaining values in metric spaces}\label{Preliminaria}

In this section we consider families of functions $f:\left[a;b\right]\rightarrow E,$ with finite truncated  variation for any $c>0,$ even if their total variation appears to be infinite. We start with 
\begin{defin}
Let $-\infty < a <b <+\infty$ and $f:\left[a;b\right]\rightarrow E.$ The function $f$ is called {\em regulated function} if for any $s \in (a;b)$ it has left and right limits, $f(s-),$ $f(s+),$ and limits $f(a+)$ and $f(b-)$ exist.
\end{defin}
Each regulated function has at most countable number of discontinuities (it follows easily from \cite[Chapt. 2, Corollary 2.2]{Norvaisa:2011}), but the possession of this property is not sufficient for a function to be a regulated one. 

We have 
\begin{fact}\label{regulated} Let $E$ be a complete metric space, $-\infty < a <b <+\infty$ and $f:\left[a;b\right]\rightarrow E.$ The following properties are equivalent
\begin{enumerate}[(a)]
\item $f$ is regulated;
\item $f$ is an uniform limit of finite variation functions;
\item for any $c>0,$ $TV^c\left(f,[a;b]\right) < + \infty.$
\end{enumerate}
\end{fact}
\begin{proof}
To prove (a)$\Rightarrow$(b) it is enough to notice that by \cite[Chapt. 2, Theorem 2.1]{Norvaisa:2011}) $f$ is an uniform limit of step functions, which have finite total variation (the assumption of \cite[Chapt. 2, Theorem 2.1]{Norvaisa:2011} that $E$ is a Banach space may be weakened and the proof follows when $E$ is a complete metric space). To prove (b)$\Rightarrow$(a) it is enough to notice that condition (b) of \cite[Chapt. 2, Theorem 2.1]{Norvaisa:2011}) holds for any function which is an uniform limit of finite variation functions. 

The implication (b)$\Rightarrow$(c) follows immediately from the inequality (\ref{eq:lower_bound}) and to prove (c)$\Rightarrow$(b) it is enough to notice that every function satisfying (c) satisfies also condition (b) of \cite[Chapt. 2, Theorem 2.1]{Norvaisa:2011}).
\end{proof}
\begin{rem}
When $E$ is not a complete metric space then the families of functions satisfying conditions (b) and (c) of Fact \ref{regulated} are still equal and contain the family of regulated functions (the implications (a)$\Rightarrow$(b) and (b)$\Rightarrow$(c) in the proof of \cite[Chapt. 2, Theorem 2.1]{Norvaisa:2011} hold) but they may be strictly greater. To see this it is enough to see that the function $f:[0;2]\rightarrow [0;1)$ such that $f(x)=x 1_{x<1}$ is not regulated for $E=[0;1)$ with standard Euclidean metric, but it has finite total variation. 
\end{rem}

\begin{rem}
From (\ref{eq:tvequality}) we may derive some upper bound for 
\[
\inf_{g\in B\left(f,c/2\right)}TV\left(g,\left[a;b\right]\right)
\]
when $f$ is c\`{a}dl\`{a}g and $E=\mathbb{R}^{N}$ with $\rho$ induced by the $L^{1}$ norm. Namely, for $f\left(t\right)=\left(f_{1}\left(t\right),\ldots,f_{N}\left(t\right)\right)\in\mathbb{R}^{N},$
$\left\Vert f\left(t\right)\right\Vert _{1}:=\left|f_{1}\left(t\right)\right|+\ldots+\left|f_{N}\left(t\right)\right|$ and $\rho(f(t),g(t)) := \left\Vert f(t) - g(t) \right\Vert _{1},$
we have 
\begin{align*}
\inf_{g\in B\left(f,c/2\right)}TV\left(g,\left[a;b\right]\right) & \leq\inf_{c_{1},\ldots,c_{N}>0,c_{1}+\ldots+c_{N}=c}\sum_{i=1}^{N}\inf_{g_{i}\in B\left(f_{i},c_{i}/2\right)}TV\left(g_{i},\left[a;b\right]\right)\\
& =\inf_{c_{1},\ldots,c_{N}>0,c_{1}+\ldots+c_{N}=c}\sum_{i=1}^{N}TV^{c_{i}}\left(f_{i},\left[a;b\right]\right).
\end{align*}
Some other upper bound for $\inf_{g\in B\left(f,c/2\right)}TV\left(g,[a;b]\right)$ was given by \cite[Theorem 10 and Theorem 11]{Vladimirov:2000}. 
\end{rem}

\section{Solution of the first and the second problem for real c\`{a}dl\`{a}g functions} \label{construction}

\subsection{Definitions and notation}

In this subsection we introduce definitions and notation which will be used
throughout the whole paper.

Let $f:\left[ a;b\right] \rightarrow \mathbb{R}$ be a c\`{a}dl\`{a}g function. For $c>0$ we define two stopping times 
\begin{gather*}
T_{D}^{c}f=\inf \left\{ s\geq a:\sup_{t\in \left[ a;s\right] }f\left(
t\right) -f\left( s\right) \geq c\right\} , \\
T_{U}^{c}f=\inf \left\{ s\geq a:f\left( s\right) -\inf_{t\in \left[ a;s%
\right] }f\left( t\right) \geq c\right\} .
\end{gather*}

Assume that $T_{D}^{c}f\geq T_{U}^{c}f$ i.e. the first upward jump of
function $f$ of size $c$ appears before the first downward jump of the same
size $c$ or both times are infinite (there is no upward neither downward jump of
size $c$). Note that in the case $T_{D}^{c}f<T_{U}^{c}f$ we may simply
consider function $-f.$ Now we define sequences $\left( T_{U,k}^{c}\right)
_{k=0}^{\infty },\left( T_{D,k}^{c}\right) _{k=-1}^{\infty },$ in the
following way: $T_{D,-1}^{c}=a,$ $T_{U,0}^{c}=T_{U}^{c}f$ and for $%
k=0,1,2,...$%
\begin{gather*}
T_{D,k}^{c}=\left\{ 
\begin{array}{lr}
\inf \left\{ s\in \left[ T_{U,k}^{c};b\right] :\sup_{t\in \left[
T_{U,k}^{c};s\right] }f\left( t\right) -f\left( s\right) \geq c\right\} & 
\text{ if }T_{U,k}^{c}<b, \\ 
\infty & \text{ if }T_{U,k}^{c}\geq b,
\end{array}
\right. \\
T_{U,k+1}^{c}=\left\{ 
\begin{array}{lr}
\inf \left\{ s\in \left[ T_{D,k}^{c};b\right] :f\left( s\right) -\inf_{t\in 
\left[ T_{D,k}^{c};s\right] }f\left( t\right) \geq c\right\} & \text{ if }%
T_{D,k}^{c}<b, \\ 
\infty & \text{ if }T_{D,k}^{c}\geq b.
\end{array}
\right. 
\end{gather*}
\begin{rem}
\label{finK} 
Times $T_{U,k}^{c}$ and $T_{D,k}^{c}$ may be seen as the consecutive times of "switching" from the two disjoint borders $\left\{(t, f(t)- c/2):t \in[a;b]\right\},$ and $\left\{(t, f(t)+ c/2):t \in[a;b]\right\}$ of the graph of a lazy function, which changes its value only if it is necessary for the relation $\left\| f-f^{c}\right\| _{\infty }\leq c/2$ to hold. 

Note that there exists such $K<\infty $ that $T_{U,K}^{c}=\infty $ or $T_{D,K}^{c}=\infty .$ Otherwise we would obtain two
infinite sequences $\left( s_{k}\right) _{k=1}^{\infty },\left( S_{k}\right)
_{k=1}^{\infty }$ such that $a\leq s_{1}<S_{1}<s_{2}<S_{2}<...$ $\leq b$ and 
$f\left( S_{k}\right) -f\left( s_{k}\right) \geq c/2.$ But this is
a contradiction, since $f$ is a c\`{a}dl\`{a}g function and $\left( f\left(
s_{k}\right) \right) _{k=1}^{\infty },\left( f\left( S_{k}\right) \right)
_{k=1}^{\infty }$ have a common limit.
\end{rem}
Now let us define for such $k$ that $T_{D,k-1}^{c}<\infty $
and $T_{U,k}^{c}<\infty $ two sequences of non-decreasing functions $m_{k}^{c}:\left[
T_{D,k-1}^{c};T_{U,k}^{c}\right) \cap \lbrack a;b]\rightarrow \mathbb{R}$
and $M_{k}^{c}:\left[ T_{U,k}^{c};T_{D,k}^{c}\right) \cap \lbrack
a;b]\rightarrow \mathbb{R}$ with the formulae 
\begin{equation*}
m_{k}^{c}\left( s\right) =\inf_{t\in \left[ T_{D,k-1}^{c};s\right] }f\left(
t\right) ,^{{}}M_{k}^{c}\left( s\right) =\sup_{t\in \left[ T_{U,k}^{c};s%
\right] }f\left( t\right) .
\end{equation*}

Next we define two finite sequences of real numbers $\left( m_{k}^{c}\right) 
$ and $\left( M_{k}^{c}\right) ,$ for such $k$ that $T_{D,k-1}^{c}<\infty $
and $T_{U,k}^{c}<\infty $\ respectively, with the formulae 
\begin{eqnarray*}
m_{k}^{c} &=&m_{k}^{c}\left( T_{U,k}^{c}-\right) =\inf_{t\in \left[
T_{D,k-1}^{c};T_{U,k}^{c}\right) \cap [a;b] }f\left( t\right) , \\
M_{k}^{c} &=&M_{k}^{c}\left( T_{D,k}^{c}-\right) =\sup_{t\in \left[
T_{U,k}^{c};T_{D,k}^{c}\right)\cap [a;b] }f\left( t\right) .
\end{eqnarray*}

\subsection{Solution of the first problem}

\label{sol1problem}

In this subsection we will solve the following problem: \emph{what
is the smallest possible\ (or infimum of) total variation of functions from
the ball $\left\{ g:\left\| f-g\right\| _{\infty }\leq c/2\right\}
?$ }

In order to solve this problem we start with results concerning
c\`{a}dl\`{a}g functions. We apply the definitions of the previous
subsection to the function $f$ and assume that $T_{D}^{c}f\geq T_{U}^{c}f.$
Define the function $f^{c}:\left[ a;b\right] \rightarrow \mathbb{R}$ with
the formulae 
\begin{equation*}
f^{c}\left( s\right) =\left\{ 
\begin{array}{lr}
m_{0}^{c}+c/2 & \text{ if }s\in \left[ a;T_{U,0}^{c}\right) ; \\ 
M_{k}^{c}\left( s\right) -c/2 & \text{ if }s\in \left[
T_{U,k}^{c};T_{D,k}^{c}\right) ,k=0,1,2,...; \\ 
m_{k+1}^{c}\left( s\right) +c/2 & \text{ if }s\in \left[
T_{D,k}^{c};T_{U,k+1}^{c}\right) ,k=0,1,2,....
\end{array}
\right.
\end{equation*}

\begin{rem}
Note that due to Remark \ref{finK}, $b$ belongs to one of the intervals $%
\left[ T_{U,k}^{c};T_{D,k}^{c}\right) $ or $\left[ T_{D,k}^{c};T_{U,k+1}^{c}%
\right) $ for some $k=0,1,2,...$ and the function $f^{c}$ is defined for
every $s\in \lbrack a;b].$
\end{rem}

\begin{rem}
One may think about the function $f^{c}$ as of the laziest function
possible, which changes its value only if it is necessary for the relation $%
\left\| f-f^{c}\right\| _{\infty }\leq c/2$ to hold. Its starting value is such that is stays in the interval $[f(t)-c/2;f(t)+c/2]$ for the longest time possible.
\end{rem}

\begin{rem}
In the case $T_{D}^{c}f<T_{U}^{c}f$ we may apply the definitions of the
previous subsection to the function $-f$ and simply define $f^{c}=-(-f)^{c}.$
Thus we will assume that the mapping $f\mapsto f^{c}$ is defined for any 
c\`{a}dl\`{a}g function. Similarly, in all the proofs of this section we
will assume $T_{D}^{c}f\leq T_{U}^{c}f,$ but all results of this section
(i.e. Lemma \ref{lem1}, Theorem \ref{thm1}, Corollary \ref{cor1}, Lemma \ref
{lem2}, Theorem \ref{thm2}, Corollary \ref{cor2} and Theorem \ref{THMM})
apply to any c\`{a}dl\`{a}g function $f.$ Obvious modifications are only
necessary in the definition of the stopping times $T_{U,k}^{c}$ and $%
T_{D,k}^{c}$ and then the functions $f_{U}^{c}$ and $f_{D}^{c}$ of Theorem 
\ref{thm1}.
\end{rem}

We have the following

\begin{lem}
\label{lem1} The function $f^{c}$ uniformly approximates the function $f$ with
accuracy $c/2$ and has finite total variation. Moreover $f^{c}$ is
a c\`{a}dl\`{a}g function and every point of the discontinuity of $f^{c}$ is also a point of discontinuity of the function $f.$
\end{lem}

\begin{proof}
Let us fix $s\in \left[ a;b\right].$ We have three possibilities.
\begin{itemize}
\item  $s\in \left[ a;T_{U,0}^{c}\right) .$ In this case, since $a\leq
s<T_{U}^{c}f \leq T_{D}^{c}f,$ 
\begin{equation*}
f\left( s\right) -f^{c}\left( s\right) =f\left( s\right) -\inf_{t\in \left[
a;T_{U,0}^{c}\right) }f\left( t\right) -c/2\in \left[ -c/2;c/2\right) .
\end{equation*}

\item  $s\in \left[ T_{U,k}^{c};T_{D,k}^{c}\right) ,$ for some $k=0,1,2,...
$ In this case $M_{k}^{c}\left( s\right) -f\left( s\right) $ belongs to the
interval $\left[ 0;c\right) ,$ hence 
\begin{equation*}
f\left( s\right) -f^{c}\left( s\right) =f\left( s\right) -M_{k}^{c}\left(
s\right) +c/2\in \left( -c/2;c/2\right] .
\end{equation*}

\item  $s\in \left[ T_{D,k}^{c};T_{U,k+1}^{c}\right) $ for some $%
k=0,1,2,...$ In this case $f\left( s\right) -m_{k+1}^{c}\left( s\right) $
belongs to the interval $\left[ 0,c\right) ,$ hence 
\begin{equation*}
f\left( s\right) -f^{c}\left( s\right) =f\left( s\right) -m_{k+1}^{c}\left(
s\right) -c/2 \in \left[ -c/2;c/2\right) .
\end{equation*}
\end{itemize}

The function $f^{c}$ has finite total variation since it is non-decreasing on
the intervals $\left[ T_{U,k}^{c};T_{D,k}^{c}\right) ,k=0,1,2,...$ and
non-increasing on the intervals $\left[ T_{D,k}^{c};T_{U,k+1}^{c}\right)
,k=0,1,2,...,$ and it has finite number of jumps between these
intervals.

For a similar reason, the function $f^{c}$ has left and right limits. To see that
it is right-continuous, let us fix $s\in \left[ a;b\right] $ and notice that
by definition of $f^{c},$ for $t \in \left( s; b \right]$ sufficiently close to $s,$ 
\begin{equation*}
f^{c}\left( t\right) =\inf_{u\in \left[ s;t\right] }f^c\left( u\right)  \text{
or }f^{c}\left( t\right) =\sup_{u\in \left[ s;t\right] }f^c\left( u\right),
\end{equation*}
and the assertion follows from the right-continuity of the function $f.$

A similar argument may be applied to prove that $f^{c}$ is continuous in every
point of continuity of $f$ except the points 
$T_{U,0}^{c},T_{D,0}^{c},T_{U,1}^{c},T_{D,1}^{c},...;$ but if $s=T_{D,i}^{c}$
and $f$ is continuous at the point $s$ then it means that $f\left(
T_{U,i}^{c}-\right) =f\left( T_{U,i}^{c}\right) =\inf_{t\in \left[
T_{D,i-1}^{c};T_{U,i}^{c}\right) }f\left( t\right) +c$ and 
\begin{equation*}
f^{c}\left( T_{U,i}^{c}-\right) =\inf_{t\in \left[ T_{D,i-1}^{c};T_{U,i}^{c}%
\right) }f\left( t\right) +c/2=f\left( T_{U,i}^{c}\right) - c/2 =f^{c}\left( T_{U,i}^{c}\right) .
\end{equation*}
A similar argument applies when $s=T_{D,i}^{c}.$

\end{proof}

Since $f^{c}$ is of finite total variation, we know that there exist such
two non-decreasing functions $f_{U}^{c}$ and $f_{D}^{c}:\left[ a;b\right]
\rightarrow \left[ 0;+\infty \right) $ that $f^{c}\left( t\right)
=f^{c}\left( a\right) +f_{U}^{c}\left( t\right) -f_{D}^{c}\left( t\right) .$

Let us examine the signs of the jumps of the function $f^{c}$ between intervals $\left[ T_{U,k}^{c};T_{D,k}^{c}\right) $ and $\left[
T_{D,k}^{c};T_{U,k+1}^{c}\right) .$ Due to c\`{a}dl\`{a}g property we have 
\begin{eqnarray*}
f^{c}\left( T_{U,k}^{c}\right) -f^{c}\left( T_{U,k}^{c}-\right)
&=&f^{c}\left( T_{U,k}^{c}\right) -m_{k}^{c}-c \\
&=&f\left( T_{U,k}^{c}\right) -\inf_{t\in \left[ T_{D,k-1}^{c};T_{U,k}^{c}%
\right) }f\left( t\right) -c\geq 0, \\
f^{c}\left( T_{D,k}^{c}\right) -f^{c}\left( T_{D,k}^{c}-\right)
&=&f^{c}\left( T_{D,k}^{c}\right) -M_{k}^{c}+2c \\
&=& f\left( T_{D,k}^{c}\right) -\sup_{t\in \left[ T_{U,k}^{c};T_{D,k}^{c}\right) }f\left(
t\right) +c\leq 0.
\end{eqnarray*}
Hence we may set $f_{U}^{c}\left( s\right) =f_{D}^{c}\left( s\right) =0$ for 
$s\in \left[ a;T_{U,0}^{c}\right) ,$ 
\begin{equation*}
f_{U}^{c}\left( s\right) =\left\{ 
\begin{array}{lr}
\sum_{i=0}^{k-1}\left\{ M_{i}^{c}-m_{i}^{c}-c\right\} +M_{k}^{c}\left(
s\right) -m_{k}^{c}-c & \text{ if }s\in \left[ T_{U,k}^{c};T_{D,k}^{c}%
\right) ; \\ 
\sum_{i=0}^{k}\left\{ M_{i}^{c}-m_{i}^{c}-c\right\} & \text{ if }s\in \left[
T_{D,k}^{c};T_{U,k+1}^{c}\right)
\end{array}
\right.
\end{equation*}
and 
\begin{equation*}
f_{D}^{c}\left( s\right) =\left\{ 
\begin{array}{lr}
\sum_{i=0}^{k-1}\left\{ M_{i}^{c}-m_{i+1}^{c}-c\right\} & \text{ if }s\in 
\left[ T_{U,k}^{c};T_{D,k}^{c}\right) ; \\ 
\sum_{i=0}^{k-1}\left\{ M_{i}^{c}-m_{i+1}^{c}-c\right\}
+M_{k}^{c}-m_{k+1}^{c}\left( s\right) -c & \text{ if }s\in \left[
T_{D,k}^{c};T_{U,k+1}^{c}\right) .
\end{array}
\right.
\end{equation*}

Now we will prove the following

\begin{thm}
\label{thm1} If $g:\left[ a;b\right] \rightarrow \mathbb{R}$ uniformly
approximates $f$ with accuracy $c/2,$ has finite total variation
and $g_{U},g_{D}:\left[ a;b\right] \rightarrow \left[ 0;+\infty \right) $
are such two non-decreasing functions that $g\left( t\right) =g\left(
a\right) +g_{U}\left( t\right) -g_{D}\left( t\right) ,t\in \left[ a;b\right]
,$ then for any $s\in \left[ a;b\right] $%
\begin{equation}
g_{U}\left( s\right) \geq f_{U}^{c}\left( s\right) \text{ and }g_{D}\left(
s\right) \geq f_{D}^{c}\left( s\right) .  \label{thm3:eq}
\end{equation}
\end{thm}

\begin{proof}
Again, we consider three cases.

\begin{itemize}
\item  $s\in \left[ a;T_{U,0}^{c}\right) .$ In this case $g_{U}\left(
s\right) \geq 0=f_{U}^{c}\left( s\right) $ as well as $g_{D}\left( s\right)
\geq 0=f_{D}^{c}\left( s\right) $

\item  $s\in \left[ T_{U,k}^{c};T_{D,k}^{c}\right) ,$ for some $k=0,1,2,...
$ In this case, from the fact that $g$ uniformly approximates $f$ with
accuracy $c/2$ and from the fact that $g_{U},g_{D}$ are non-decreasing, for $i=0,1,2,...k-1$ we get 
\begin{align*}
& \sup_{s_{i}\in \left[ T_{U,i}^{c};T_{D,i}^{c}\right) }g_{U}\left(
s_{i}\right) -\inf_{s_{i}\in \left[ T_{D,i-1}^{c};T_{U,i}^{c}\right)
}g_{U}\left( s_{i}\right)  \\
& \geq \sup_{s_{i}\in \left[ T_{U,i}^{c};T_{D,i}^{c}\right) }\left(
g_{U}-g_{D}\right) \left( s_{i}\right) -\inf_{s_{i}\in \left[
T_{D,i-1}^{c};T_{U,i}^{c}\right) }\left( g_{U}-g_{D}\right) \left(
s_{i}\right)  \\
& =\sup_{s_{i}\in \left[ T_{U,i}^{c};T_{D,i}^{c}\right) }g\left(
s_{i}\right) -\inf_{s_{i}\in \left[ T_{D,i-1}^{c};T_{U,i}^{c}\right)
}g\left( s_{i}\right)  \\
& \geq \sup_{s_{i}\in \left[ T_{U,i}^{c};T_{D,i}^{c}\right) }\left\{
f\left( s_{i}\right) -c/2\right\} -\inf_{s_{i}\in \left[
T_{D,i-1}^{c};T_{U,i}^{c}\right) }\left\{ f\left( s_{i}\right) +c/2\right\} 
\\
& =M_{i}^{c}-m_{i}^{c}-c.
\end{align*}
Similarly 
\begin{align*}
& g_{U}\left( s\right) -\inf_{s_{k}\in \left[ T_{D,k-1}^{c};T_{U,k}^{c}%
\right) }g_{U}\left( s_{k}\right)  \\
& =\sup_{t\in \left[ T_{U,k}^{c};s\right] }g_{U}\left( t\right)
-\inf_{s_{k}\in \left[ T_{D,k-1}^{c};T_{U,k}^{c}\right) }g_{U}\left(
s_{k}\right)  \\
& \geq \sup_{t\in \left[ T_{U,k}^{c};s\right] }\left( g_{U}-g_{D}\right)
\left( t\right) -\inf_{s_{k}\in \left[ T_{D,k-1}^{c};T_{U,k}^{c}\right)
}\left( g_{U}-g_{D}\right) \left( s_{k}\right)  \\
& =\sup_{t\in \left[ T_{U,k}^{c};s\right] }g\left( t\right) -\inf_{s_{k}\in %
\left[ T_{D,k-1}^{c};T_{U,k}^{c}\right) }g\left( s_{k}\right)  \\
& \geq \sup_{t\in \left[ T_{U,k}^{c};s\right] }\left\{ f\left( t\right)
-c/2\right\} -\inf_{s_{k}\in \left[ T_{D,k-1}^{c};T_{U,k}^{c}\right)
}\left\{ f\left( s_{k}\right) +c/2\right\}  \\
& =M_{k}^{c}\left( s\right) -m_{k}^{c}-c.
\end{align*}
Summing up the above inequalities and using monotonicity of $g_{U}$\ we
finally get 
\begin{equation*}
g_{U}\left( s\right) \geq \sum_{i=0}^{k-1}\left\{
M_{i}^{c}-m_{i}^{c}-c\right\} +M_{k}^{c}\left( s\right)
-m_{k}^{c}-c=f_{U}^{c}\left( s\right) .
\end{equation*}

The proof of the corresponding inequality for $g_{D}$ follows similarly and
we get
\begin{equation*}
g_{D}\left( s\right) \geq \sum_{i=0}^{k-1}\left\{
M_{i}^{c}-m_{i+1}^{c}-c\right\} =f_{D}^{c}\left( s\right) .
\end{equation*}

\item  $s\in \left[ T_{D,k}^{c};T_{U,k+1}^{c}\right) $
The proof follows similarly as in the previous case.
\end{itemize}
\end{proof}

From Theorem \ref{thm1} we immediately get that the decomposition 
\begin{equation}
f^{c}\left( s\right) =f^{c}\left( a\right) +f_{U}^{c}\left( s\right)
-f_{D}^{c}\left( s\right)  \label{decomp}
\end{equation}
is minimal (cf. \cite{Revuz:1991kx}, page 5) thus the total variation of the
function $f^{c}$ on the interval $\left[ a;s\right] $ equals $%
f_{U}^{c}\left( s\right) +f_{D}^{c}\left( s\right) .$ 
\begin{rem}
\label{fUfDcadlag} From Lemma \ref{lem1} and the minimality of the decomposition
(\ref{decomp}) it follows that $f_{U}^{c}$ and $f_{U}^{c}$ are also c\`{a}dl\`{a}g 
functions and that every point of their discontinuity is also a point
of discontinuity of the function $f.$ Moreover, due to the minimality of the variation of the 
function $f^c,$ any jump of $f^c$ is no greater than the jump of the function $f.$
\end{rem}
We also have 
\begin{corollary}
\label{cor1} The function $f^{c}$ is optimal i.e. if $g:\left[ a;b\right]
\rightarrow \mathbb{R}$ is such that $\left\| f-g\right\| _{\infty }\leq 
c/2$ and has finite total variation, then for every $s\in \left[ a;b\right] $ 
\begin{equation*}
TV\left( g,\left[ a;s\right] \right) \geq TV\left( f^{c},\left[ a;s\right]
\right) .
\end{equation*}
Moreover, it is unique in such a sense that if for every $s\in \left[ a;b\right] $ 
the opposite inequality holds 
\begin{equation*}
TV\left( g,\left[ a;s\right] \right) \leq TV\left( f^{c},\left[ a;s\right]
\right) 
\end{equation*}
and $c\leq \sup_{s,u\in \lbrack a;b]}|f(s)-f(u)|$ then $g=f^{c}.$
\end{corollary}

\begin{proof}
Let $g_{U},g_{D}:\left[ a;b\right] \rightarrow \left[ 0;+\infty \right) $ be
two non-decreasing functions such that for $s\in \left[ a;b\right] ,$ 
$g\left( s\right) =g\left( a\right) +g_{U}\left( s\right) -g_{D}\left(
s\right) $ and $TV\left( g, \left[ a;s\right]\right)  =g_{U}\left( s\right)
+g_{D}\left( s\right) .$ 

The first assertion follows directly from Theorem \ref{thm1} and the fact that $TV\left( g, \left[ a;s\right]\right)  =g_{U}\left( s\right)
+g_{D}\left( s\right) .$ 

The opposite inequality, $TV\left( g,\left[ a;s\right] \right) \leq TV\left( f^{c},\left[ a;s\right]
\right),$ holds for every $s\in \left[ a;b\right] $\ iff 
$g_{U}\left( s\right) =f_{U}^{c}\left( s\right) $ and $g_{D}\left( s\right)
=f_{D}^{c}\left( s\right) .$ Thus in such a case we get $g\left( s\right)
-f^{c}\left( s\right) =g\left( a\right) -f^{c}\left( a\right) $ and we have 
\begin{eqnarray}
c/2 &\geq &\inf_{s\in \left[ a;T_{U,0}^{c}\right) }\left\{ g\left( s\right)
-f\left( s\right) \right\} =\inf_{s\in \left[ a;T_{U,0}^{c}\right) }\left\{
g\left( a\right) -f^{c}\left( a\right) +f^{c}\left( s\right) -f\left(
s\right) \right\}   \notag \\
&=&g\left( a\right) -f^{c}\left( a\right) +c/2  \label{g1}
\end{eqnarray}
(notice that  $T_{U,0}^{c} \leq b$ since $c \leq \sup_{s,u \in [a;b]} | f(s) - f(u) |$ and $T_{U,0}^{c} \leq T_{D,0}^{c}$).
On the other hand we have 
\begin{eqnarray}
-c/2 &\leq &g\left( T_{U,0}^{c}\right) -f\left( T_{U,0}^{c}\right) =g\left(
a\right) -f^{c}\left( a\right) +f^{c}\left( T_{U,0}^{c}\right) -f\left(
T_{U,0}^{c}\right)   \notag \\
&=&g\left( a\right) -f^{c}\left( a\right) -c/2.  \label{g2}
\end{eqnarray}
From (\ref{g1}) and (\ref{g2}) we get $g\left( a\right) =f^{c}\left(
a\right) .$ This together with the equalities $g_{U}\left( s\right)
=f_{U}^{c}\left( s\right) $ and $g_{D}\left( s\right) =f_{D}^{c}\left(
s\right) $ gives $g=f^{c}.$
\end{proof}
\begin{rem}
The formula obtained for the smallest possible total variation of a function
from the ball $\left\{ g:\left\| f-g\right\|_{\infty} \leq c/2\right\} $ reads as 
\begin{equation*}
f_{U}^{c}\left( b\right) +f_{D}^{c}\left( b\right)
\end{equation*}
and does not resemble formula (\ref{eq:tvequality}). In subsection \ref
{relationutvdtv} we will show that these formulae coincide.
\end{rem}

\subsection{Solution of the second problem}

In this subsection we will solve the following problem: \emph{\ for a
c\`{a}dl\`{a}g function $f:\left[ a;b\right] \rightarrow \mathbb{R}$ and $c>0
$ find 
\begin{equation*}
\inf \left\{ TV\left( f+h,\left[ a;b\right] \right) :\left\| h\right\|
_{osc}\leq c\right\} ,
\end{equation*}
where $h:\left[ a;b\right] \rightarrow \mathbb{R},$ $\left\| h\right\| _{osc}:=\sup_{s,u\in \left[ a;b\right] }\left|
h\left( s\right) -h\left( u\right) \right| .$}

We will show that 
\begin{equation*}
\inf \left\{ TV\left( f+h,\left[ a;b\right] \right) :\left\| h\right\|
_{osc}\leq c\right\} =f_{U}^{c}\left( b\right) +f_{D}^{c}\left( b\right) ,
\end{equation*}
where $f_{U}^{c}$ and $f_{D}^{c}$ were defined in the previous subsection.
In order to do it let us simply define 
\begin{equation*}
f^{0,c}=f_{U}^{c}-f_{D}^{c}.
\end{equation*}
We have

\begin{lem}
\label{lem2} The increments of the function $f^{0,c}$ uniformly approximate the
increments of the function $f$ with accuracy $c$ and the function $f^{0,c}$ has
finite total variation.
\end{lem}

\begin{proof}
Since the difference $f^c - f^{0,c}$ is constant, the first and the second assertion follows immediately from 
Lemma \ref{lem1} and from simple calculation that for any $s,u \in [a;b],$
\begin{eqnarray*}
\lefteqn{ \left\{ f^{0,c}\left( s\right) -f^{0,c}\left( u\right) \right\}- \left\{ f\left( s\right) -f\left( u\right) \right\} }\\
& = &\left\{ f^{c}\left( s\right) -f\left( s\right) \right\} -\left\{
f^{c}\left( u\right) -f\left( u\right) \right\} \in [-c;c].
\end{eqnarray*}
\end{proof}

Now we will prove the analog of Theorem \ref{thm1}.

\begin{thm}
\label{thm2} If the increments of the function $g:\left[ a;b\right] \rightarrow 
\mathbb{R}$ uniformly approximate the increments of the function $f$ with
accuracy $c,$ $g$ has finite total variation and $g_{U},g_{D}:\left[
a;b\right] \rightarrow \left[ 0;+\infty \right) $ are such two
non-decreasing functions that $g\left( t\right) =g\left( a\right)
+g_{U}\left( t\right) -g_{D}\left( t\right) ,t\in \left[ a;b\right] ,$ then
for any $s\in \left[ a;b\right] $%
\begin{equation*}
g_{U}\left( s\right) \geq f_{U}^{c}\left( s\right) \text{ and }g_{D}\left(
s\right) \geq f_{D}^{c}\left( s\right) .
\end{equation*}
\end{thm}

\begin{proof}
It is enough to see that for $h = g - f,$ $\left\| h\right\|
_{osc}\leq c,$ thus for $${\alpha} = - \frac{1}{2} \left\{ \inf_{s \in [a;b] }h(s) + \sup_{s \in [a;b] }h(s) \right\},$$ $\left\| {\alpha} + h\right\|
_{\infty}\leq \tfrac{1}{2}c,$ and the function $g_{{\alpha}} = {\alpha} + g$ belongs to the ball $\left\{ g:\left\| f-g\right\| _{\infty }\leq \tfrac{1}{2}c\right\}.$ Application of Theorem \ref{thm1} to the function $g_{{\alpha}}$ concludes the proof.

\end{proof}

Since the decomposition $f^{0,c}\left( s\right) =$ $f_{U}^{c}\left( s\right)
-f_{D}^{c}\left( s\right) $ is minimal and $f^{0,c}\left( a\right) = 0$ we
immediately obtain

\begin{corollary}
\label{cor2} The function $f^{0,c}$ is optimal i.e. if $g:\left[ a;b\right]
\rightarrow \mathbb{R}$ is such that 
\begin{equation*}
\sup_{a\leq u<s\leq b}\left| \left\{ g\left( s\right) -g\left( u\right)
\right\} -\left\{ f\left( s\right) -f\left( u\right) \right\} \right| \leq c
\end{equation*}
and $g$ has finite total variation, then for every $s\in \left[ a;b\right] $ 
\begin{equation*}
TV\left( g,\left[ a;s\right] \right) \geq TV\left( f^{0,c},\left[ a;s\right]
\right) .
\end{equation*}
Moreover, it is unique in such a sense that if $g\left( a\right) =0$ and\
for every $s\in \left[ a;b\right] $ the opposite inequality holds 
\begin{equation*}
TV\left( g,\left[ a;s\right] \right) \leq TV\left( f^{0,c},\left[ a;s\right]
\right) ,
\end{equation*}
then $g=f^{0,c}.$
\end{corollary}

From Corollary \ref{cor2} it immediately follows that 
\begin{equation*}
\inf \left\{ TV\left( f+h,\left[ a;b\right] \right) :\left\| h\right\|
_{osc}\leq c\right\} =f_{U}^{c}\left( b\right) +f_{D}^{c}\left( b\right).
\end{equation*}
Indeed, for any $h$ such that $\left\| h\right\| _{osc} \leq c$ we put $g =
f + h$ and if $g$ has finite total variation then it satisfies the
assumptions of Corollary \ref{cor2} and we get 
\begin{equation*}
TV\left( g,\left[ a;b\right] \right) \geq TV\left( f^{0,c},\left[ a;b\right]
\right) = f_{U}^{c}\left( b\right) +f_{D}^{c}\left( b\right) .
\end{equation*}

\section{Relation of the solutions of the first and the second problem
with truncated variation, upward truncated variation and downward truncated
variation}
\label{relationutvdtv}

In order to prove (\ref{eq:tvequality}), (\ref{tvceq}) and (\ref{sumutvdtv}), where $UTV^{c}\left( f,\left[ a;s\right] \right) $ and $DTV^{c}\left( f,\left[ a;s \right] \right) $ are defined by (\ref{utv:def}) and (\ref{dtv:def})
respectively,\ it is enough to prove

\begin{thm}
\label{THMM} For a given c\`{a}dl\`{a}g function $f:\left[ a;b\right]
\rightarrow \mathbb{R}$ and for any $s\in \left( a;b\right] $\ the following
equalities hold 
\begin{gather}
UTV^{c}\left( f,\left[ a;s\right] \right) =f_{U}^{c}\left( s\right) ,
\label{UTVfU} \\
DTV^{c}\left( f,\left[ a;s\right] \right) =f_{D}^{c}\left( s\right) ,
\label{DTVfD} \\
TV^{c}\left( f,\left[ a;s\right] \right) =f_{U}^{c}\left( s\right)
+f_{D}^{c}\left( s\right) .  \label{eq3}
\end{gather}
\end{thm}

\begin{proof}
Examining (with obvious modifications) the proof of Lemma 3 from \cite{Lochowski:2011}, we see that it may be applied to the c\`{a}dl\`{a}g (but not necessarily continuous)
function $f$\ and we obtain 
\begin{equation}
UTV^{c}\left( f, \left[ a;s\right] \right) =\sup_{a\leq t<u\leq \left(
T_{D}^{c}f\right) \wedge s}\left( f\left( u\right) -f\left( t\right)
-c\right) _{+}+UTV^{c}\left( f, \left[ \left( T_{D}^{c}f\right) \wedge
s;s\right] \right) .  \label{utvf}
\end{equation}
Now, from the assumption $T_{D}^{c}f\geq T_{U}^{c}f$ we get 
$T_{D}^{c}f=T_{D,0}^{c}$ and we have that 
\begin{equation*}
\sup_{a\leq t<u\leq \left( T_{D}^{c}f\right) \wedge s}\left( f\left(
u\right) -f\left( t\right) -c\right) _{+}=\left\{ 
\begin{array}{lr}
0 & \text{ if }s\in \left[ a;T_{U,0}^{c}\right) ; \\ 
M_{0}^{c}\left( s\right) -m_{0}^{c}-c & \text{ if }s\in \left[
T_{U,0}^{c};T_{D,0}^{c}\right) ; \\ 
M_{0}^{c}-m_{0}^{c}-c & \text{ if }s\geq T_{D,0}^{c}.
\end{array}
\right. 
\end{equation*}
Iterating the equality (\ref{utvf}) we obtain 
\begin{eqnarray*}
UTV^{c}\left( f, \left[ a;s\right]  \right) &=&\left\{ 
\begin{array}{lr}
0 & \text{ if }s\in \left[ a;T_{U,0}^{c}\right) ; \\ 
\sum_{i=0}^{k-1}\left( M_{i}^{c}-m_{i}^{c}-c\right) +M_{k}^{c}\left(
s\right) -m_{k}^{c}-c & \text{ if }s\in \left[ T_{U,k}^{c};T_{D,k}^{c}\right) ;
\\ 
\sum_{i=0}^{k}\left( M_{i}^{c}-m_{i}^{c}-c\right) & \text{ if }s\in \left[
T_{D,k}^{c};T_{U,k+1}^{c}\right) 
\end{array}
\right.  \\
&=&f_{U}^{c}\left( s\right) .
\end{eqnarray*}
\begin{rem}
Iterating (\ref{utvf}) we obtain a bit different equality than $UTV^{c}\left( f, \left[ a;s\right]  \right) = f_{U}^{c} $ , but equivalent with it. To see this let us define the following sequence of times.  $\tilde{T}_{D,-1}^{c}=a,$ and for $
k=0,1,2,...$%
\begin{equation*}
\tilde{T}_{D,k+1}^{c}=\inf \left\{ s>\tilde{T}_{D,k}^{c}:\sup_{t\in \left[ 
\tilde{T}_{D,k}^{c};s\right] }f\left( t\right) -f\left( s\right) \geq c\right\} .
\end{equation*}
Let us fix $s_{0} \in [a;b]$ and define $k_{0}=\max \left\{ k:\tilde{T}%
_{D,k}^{c}\leq s_{0}\right\} .$ Iterating (\ref{utvf}) we obtain the following equality 
\begin{equation*}
UTV^{c}\left( f,  \left[ a;s_{0}\right]\right) =\sum_{k=1}^{k_{0}-1}\sup_{%
\tilde{T}_{D,k}^{c}\leq s<u\leq \tilde{T}_{D,k+1}^{c}}\left( f\left(
u\right) -f\left( s\right) -c\right) _{+}+UTV^{c}\left( f, \left[ 
\tilde{T}_{D,k_{0}}^{c};s_{0}\right]\right)  
\end{equation*}
which looks different from $f_{U}^{c}\left( s_0\right) .$ But it
is easy to notice that for all $k\geq 1$ such that $\tilde{T}%
_{D,k+1}^{c}<T_{U,1}^{c}f$ the summand $\sup_{\tilde{T}_{D,k}^{c}\leq s<u\leq 
\tilde{T}_{D,k+1}^{c}}\left( f\left( u\right) -f\left( s\right) -c\right)
_{+}$ is equal zero. Thus in fact both quantities coincide.
\end{rem}

Identically we prove that $DTV^{c}\left( f\right) \left[ a;s\right]
=f_{D}^{c}\left( s\right) .$

Now, in order to prove the equality (\ref{eq3}) simply notice that 
$TV^{c}\left( f, \left[ a;s\right]\right)  \geq 0$\ and if $s\in \left[
T_{U,k}^{c};T_{D,k}^{c}\right) $ 
\begin{eqnarray*}
TV^{c}\left( f, \left[ a;s\right]\right)   &\geq &\sum_{i=0}^{k-1}\left(
M_{i}^{c}-m_{i}^{c}-c\right) +\sum_{i=0}^{k-1}\left(
M_{i}^{c}-m_{i+1}^{c}-c\right) +M_{k}^{c}\left( s\right) -m_{k}^{c}-c \\
&=&f_{U}^{c}\left( s\right) +f_{D}^{c}\left( s\right) .
\end{eqnarray*}
Analogously, if $s\in \left[ T_{D,k}^{c};T_{U,k+1}^{c}\right) $ 
\begin{eqnarray*}
TV^{c}\left( f, \left[ a;s\right]\right)   &\geq &\sum_{i=0}^{k-1}\left(
M_{i}^{c}-m_{i}^{c}-c\right) +\sum_{i=0}^{k-1}\left(
M_{i}^{c}-m_{i+1}^{c}-c\right) +M_{k}^{c} -m_{k+1}^{c}\left(
s\right) -c \\
&=&f_{U}^{c}\left( s\right) +f_{D}^{c}\left( s\right) .
\end{eqnarray*}
Hence for all $s\in \left[ a;b\right] $ 
\begin{equation*}
TV^{c}\left( f, \left[ a;s\right]\right)  \geq f_{U}^{c}\left( s\right)
+f_{D}^{c}\left( s\right) .
\end{equation*}

So 
\begin{equation*}
TV^{c}\left( f, \left[ a;s\right]\right)  \geq UTV^{c}\left( f, \left[
a;s\right]\right)  +DTV^{c}\left( f, \left[ a;s\right]\right)  .
\end{equation*}
Since the opposite inequality is obvious, we finally get (\ref{eq3}).
\end{proof}

Now we see that by Corollary \ref{cor1} and Corollary \ref{cor2} functions $%
h^{c}=f^{c}-f$ and $h^{0,c}=f(a)+f^{0,c}-f= f(a) + UTV^{c}\left( f,\left[ a;.%
\right] \right) -DTV^{c}\left( f,\left[ a;.\right] \right) -f$ are optimal
and such that for any $s\in \left( a;b\right] $%
\begin{eqnarray*}
\inf \left\{ TV\left( f+h,\left[ a;s\right] \right) :\left\| h\right\|
_{\infty }\leq c/2\right\} &=&TV\left( f+h^{c},\left[ a;s\right]
\right) \\
&=&TV^{c}\left( f,\left[ a;s\right] \right),
\end{eqnarray*}
\begin{eqnarray*}
\inf \left\{ TV\left( f+h,\left[ a;s\right] \right) :\left\| h\right\|
_{osc}\leq c\right\} &=&TV\left( f+h^{0,c},\left[ a;s\right] \right) \\
&=&TV^{c}\left( f,\left[ a;s\right] \right) .
\end{eqnarray*}
Moreover, by Remark \ref{fUfDcadlag}, $h^{c}$ and $h^{0,c}$ are also
c\`{a}dl\`{a}g functions and every point of their discontinuity is also a
point of discontinuity of the function $f.$

\section{Further properties of truncated variation, upward truncated
variation and downward truncated variation}

In this section we summarize basic properties of the defined functionals.
We start with

\subsection{Algebraic properties.}

For any $c>0$ we have 
\begin{gather}
DTV^{c}\left( f,\left[ a;b\right] \right) =UTV^{c}\left( -f,\left[ a;b\right]
\right) ,  \label{A1} \\
TV^{c}\left( f,\left[ a;b\right] \right) =UTV^{c}\left( f,\left[ a;b\right]
\right) +DTV^{c}\left( f,\left[ a;b\right] \right) .  \label{A2}
\end{gather}
Property (\ref{A1}) follows simply from the definitions (\ref{utv:def}) and (%
\ref{dtv:def}). Property (\ref{A2}) is the consequence of Theorem \ref{THMM}.

\subsection{Properties of $UTV^{c}\left( f,\left[ a;b\right] \right)
,DTV^{c}\left( f,\left[ a;b\right] \right) $ and $TV^{c}\left( f,\left[ a;b%
\right] \right) $ as the functions of the parameter $c$}

We have the following

\begin{fact}
For any c\`{a}dl\`{a}g function $f$ the functions $\left( 0;\infty \right)
\ni c\mapsto UTV^{c}\left( f,\left[ a;b\right] \right) \in \left[ 0;+\infty
\right) ,$ $\left( 0;\infty \right) \ni c\mapsto DTV^{c}\left( f,\left[ a;b%
\right] \right) \in \left[ 0;+\infty \right) $ and $\left( 0;\infty \right)
\ni c\mapsto TV^{c}\left( f,\left[ a;b\right] \right) \in \left[ 0;+\infty
\right) $ are non-increasing, continuous, convex functions of the parameter 
$c.$ Moreover, $\lim_{c\downarrow 0}TV^{c}\left( f,\left[ a;b\right] \right)
=TV\left( f,\left[ a;b\right] \right) $ and for any $c\geq \left\| f\right\|
_{osc},$ $TV^{c}\left( f,\left[ a;b\right] \right) =0.$
\end{fact}

\begin{proof}
The finiteness of $TV,$ $UTV$ and $DTV$ follows from Lemma \ref{lem1} and Theorem \ref{THMM}. Monotonicity is obvious. 

We start with the proof of the convexity. Let us fix $c,\varepsilon >0$\ and
consider such a partition $a\leq t_{0}<t_{1}<...<t_{n}\leq b$ of the
interval $\left[ a;b\right] $ that 
\begin{equation*}
UTV^{c}\left( f,\left[ a;b\right] \right) \leq \sum_{i=0}^{n-1}\max \left\{
f\left( t_{i+1}\right) -f\left( t_{i}\right) -c,0\right\} +\varepsilon .
\end{equation*}
Taking $\alpha \in \left[ 0;1\right] $ and $c_{1},c_{2}>0$ such that $%
c=\alpha c_{1}+\left( 1-\alpha \right) c_{2}$\ we have the inequality 
\begin{multline*}
\max \left\{ f\left( t_{i+1}\right) -f\left( t_{i}\right) -\alpha
c_{1}-\left( 1-\alpha \right) c_{2},0\right\}  \\
=\max \left\{ \alpha \left( f\left( t_{i+1}\right) -f\left( t_{i}\right)
-c_{1}\right) +\left( 1-\alpha \right) \left( f\left( t_{i+1}\right)
-f\left( t_{i}\right) -c_{2}\right) ,0\right\}  \\
\leq \alpha \max \left\{ f\left( t_{i+1}\right) -f\left( t_{i}\right)
-c_{1},0\right\} +\left( 1-\alpha \right) \max \left\{ f\left(
t_{i+1}\right) -f\left( t_{i}\right) -c_{2},0\right\} .
\end{multline*}
Now 
\begin{eqnarray*}
UTV^{c}\left( f,\left[ a;b\right] \right)  &\leq &\sum_{i=0}^{n-1}\max
\left\{ f\left( t_{i+1}\right) -f\left( t_{i}\right) -c,0\right\}
+\varepsilon  \\
&\leq &\alpha \sum_{i=0}^{n-1}\max \left\{ f\left( t_{i+1}\right) -f\left(
t_{i}\right) -c_{1},0\right\}  \\
&&+\left( 1-\alpha \right) \sum_{i=0}^{n-1}\max \left\{ f\left(
t_{i+1}\right) -f\left( t_{i}\right) -c_{2},0\right\} +\varepsilon  \\
&\leq &\alpha UTV^{c_{1}}\left( f,\left[ a;b\right] \right) +\left( 1-\alpha
\right) UTV^{c_{2}}\left( f,\left[ a;b\right] \right) +\varepsilon .
\end{eqnarray*}
Since $\varepsilon $ may be arbitrary small, we obtain the convexity
assertion. From convexity and monotonicity we obtain the continuity
assertion.

The same properties of $DTV$ and $TV$ follow immediately from (\ref{A1}) and
(\ref{A2}).

The fact that for $c\geq \left\| f\right\| _{osc},$ $TV^{c}\left( f,\left[ a;b\right] \right) =0$ 
follows easily from equality 
$$\max \left\{ \left| f\left( t_{i+1}\right) -f\left( t_{i}\right) \right| -c,0\right\} =0 $$
satisfied for any such $c$ and $t_{i}, t_{i+1} \in [a;b].$
\end{proof}

\begin{rem}
\cite[Theorem 17]{Vladimirov:2000} gives some estimates for the rate of the convergence of $TV^c\left( f,\left[ a;b\right] \right)$ to $+\infty$ when $c\downarrow 0$ and $f$ has finite $p$-variation with $p > 1$. 
\end{rem}

\section*{Acknowledgements}
The author would like to express his gratitude to Prof. Przemys\l aw
Woj-taszczyk from Warsaw University for very helpful conversations which
facilitated the finding of the solutions of the two problems solved in Section \ref{construction} and to Prof. Rimas Norvai\v{s}a from Vilnius University for
pointing out to him the notion of regulated functions. 

This research was partly supported by the National Science Centre in Poland under the decision no. DEC-2011/01/B/ST1/05089.

\end{document}